\documentclass[12pt,reqno,draft]{amsart}
\usepackage{amsthm}
\usepackage{amscd}
\usepackage{amsfonts}
\usepackage{amssymb}
\usepackage{amsgen}
\usepackage{amsmath}
\usepackage{amsopn}
\usepackage{verbatim}
\usepackage{xypic}
\usepackage{xspace}
\usepackage{multicol}
\usepackage{url}
\usepackage{upref}

\theoremstyle{plain}
\newtheorem{thm}{Theorem}[section]
\newtheorem{lem}[thm]{Lemma}

\theoremstyle{definition}

\newtheorem{defn}[thm]{Definition}

\theoremstyle{remark}
\newtheorem{rem}[thm]{Remark}

\newtheorem{eg}[thm]{Example}

 \font\cyr=wncyr10
 \newcommand{\nc}{\newcommand}

\DeclareMathOperator{\Trunc}{{Tr}}

\DeclareMathOperator{\sgn}{{sgn}}

\nc{\per}[1]{\underset{#1}{\boldsymbol \pi}\,}

 \nc{\MT}{{\rm MT}}
 \nc{\XX}{{X}}
 \nc{\gF}{{\varPhi}}
 \nc{\ot}{\otimes}

 \nc{\wht}{\widehat}
 \nc{\bwg}{{\bigwedge}}
 \nc{\wg}{{\wedge}}
 \nc{\mmu}{{\boldsymbol{\mu}}}
 \nc{\mal}{{{\scriptstyle \maltese}}}
 \nc{\fA}{{\mathfrak A}}
 \nc{\HH}{{\mathfrak H}}
 \nc{\ra}{\rightarrow}
 \nc{\ors}{{\bfs}}
 \nc{\orr}{{\bfr}}
 \nc{\os}{{\overset}}
 \nc{\G}{{\mathbb G}}
 \nc{\F}{{\mathbb F}}
 \nc{\Z}{{\mathbb Z}}
 \nc{\R}{{\mathbb R}}
 \nc{\N}{{\mathbb N}}
 \nc{\ZN}{{\mathbb Z_{\ge 0}}}
 \nc{\Q}{{\mathbb Q}}
 \nc{\C}{{\mathbb C}}
 \nc{\CP}{{\mathbb{CP}}}
 \nc{\Cnn}{{\mathbb C}_{\ge 0}}
 \nc{\Cp}{{\mathbb C}_{>0}}
 \nc{\MPV}{{\mathcal{MPV}}}
 
 \nc{\tB}{{\tilde B}}

 \nc{\oI}{{\overline{I}}}
 \nc{\bI}{{\bar{I}}}

 \nc{\suf}{{\ast\,}}
 \nc{\sufq}{{\ast_q\,}}
 \nc{\gam}{{\gamma}}
 \nc{\gG}{{\Gamma}}
 \nc{\om}{{\omega}}
 \nc{\vep}{{\varepsilon}}
 \nc{\ga}{{\alpha}}
 \nc{\gl}{{\lambda}}
 \nc{\gb}{{\beta}}
 \nc{\gf}{{\varphi}}
 \nc{\gd}{{\delta}}
 \nc{\orgd}{{\vec \gd\,}}
 \nc{\gs}{{\sigma}}
 \nc{\gth}{{\theta}}
 \nc{\gS}{{\Sigma}}

 \nc{\gk}{{\kappa}}
  \nc{\gz}{{\zeta}}
 \nc{\tgz}{{\tilde{\zeta}}}
 \nc{\gO}{{\Omega}}
 \nc{\sif}{{\mathcal S}}
 \nc{\gt}{{\tau}}
 \nc{\Lra}{\Longrightarrow}
 \nc{\lra}{\longrightarrow}
 \nc{\lmaps}{\longmapsto}
 \nc{\fS}{{\mathfrak S}}
 \nc{\DD}{{\mathfrak D}}
 \nc{\Llra}{\Longleftrightarrow}
 \nc{\ol}{\overline}
 \nc{\ola}{\overleftarrow}
 \nc{\lms}{\longmapsto}
 \nc{\cv}{{{\mathsf c}{\mathsf v}}}
 \nc{\zq}{{\zeta_q}}
 \nc\qup{{q\uparrow 1}}
 \nc{\us}{\underset}
 \nc{\tn}{{\tilde{n}}}
 \nc{\gD}{{\Delta}}
 \nc{\bi}{{\bf i}}
 \nc{\bfone}{{\bf 1}}

 \nc{\bfa}{{\bf a}}
 \nc{\bfb}{{\bf b}}
 \nc{\bfc}{{\bf c}}
 \nc{\bfd}{{\bf d}}
 \nc{\bfe}{{\bf e}}
 \nc{\bff}{{\bf f}}
 \nc{\bfg}{{\bf g}}
 \nc{\bfi}{{\bf i}}
 \nc{\bfj}{{\bf j}}

 \nc{\bfn}{{\bf n}}
 \nc{\bfl}{{\bf l}}
 \nc{\bfk}{{\bf k}}
 \nc{\bfm}{{\bf m}}
 \nc{\bfo}{{\bf o}}
 \nc{\bfp}{{\bf p}}
 \nc{\bfq}{{\bf q}}
 \nc{\bfr}{{\bf r}}
 \nc{\bfs}{{\bf s}}
 \nc{\bft}{{\bf t}}
 \nc{\bfu}{{\bf u}}
 \nc{\tbfs}{{\tilde{\bfs}}}
 \nc{\tbft}{{\tilde{\bft}}}
 \nc{\tbfu}{{\tilde{\bfu}}}
 \nc{\bfv}{{\bf v}}
 \nc{\bfw}{{\bf w}}
 \nc{\bfx}{{\bf x}}
 \nc{\bfy}{{\bf y}}
 \nc{\bfz}{{\bf z}}
 \nc{\bfB}{{\bf B}}
 \nc{\bfP}{{\bf P}}
 \nc{\bfQ}{{\bf Q}}
 \nc{\bfY}{{\bf Y}}
 \nc{\bfgb}{{\boldsymbol \gb}}
 \nc{\bfga}{{\boldsymbol \ga}}
 \nc{\bfrho}{{\boldsymbol \rho}}
 \nc{\bfchi}{{\boldsymbol \chi}}
 \nc{\QX}{{\Q\langle \bfX\rangle}}
 \nc{\QY}{{\Q\langle \bfY\rangle}}
 \nc{\CX}{{\C\langle \bfX\rangle}}
 \nc{\CY}{{\C\langle \bfY\rangle}}
 \nc{\QXX}{{\Q\langle\!\langle \bfX\rangle\!\rangle}}
 \nc{\QYY}{{\Q\langle\!\langle \bfY\rangle\!\rangle}}
 \nc{\CXX}{{\C\langle\!\langle \bfX\rangle\!\rangle}}
 \nc{\CYY}{{\C\langle\!\langle \bfY\rangle\!\rangle}}

 \nc{\bbA}{{\mathbb A}}
 \nc{\bbB}{{\mathbb B}}
 \nc{\bbC}{{\mathbb C}}
 \nc{\bbD}{{\mathbb D}}
 \nc{\bbE}{{\mathbb E}}
 \nc{\bbF}{{\mathbb F}}
 \nc{\bbG}{{\mathbb G}}
 \nc{\bbH}{{\mathbb H}}
 \nc{\bbI}{{\mathbb I}}
 \nc{\bbJ}{{\mathbb J}}
 \nc{\bbK}{{\mathbb K}}
 \nc{\bbL}{{\mathbb L}}
 \nc{\bbM}{{\mathbb M}}
 \nc{\bbN}{{\mathbb N}}
 \nc{\bbO}{{\mathbb O}}
 \nc{\bbP}{{\mathbb P}}
 \nc{\bbQ}{{\mathbb Q}}
 \nc{\bbR}{{\mathbb R}}
 \nc{\bbS}{{\mathbb S}}
 \nc{\bbT}{{\mathbb T}}
 \nc{\bbU}{{\mathbb U}}
 \nc{\bbV}{{\mathbb V}}
 \nc{\bbW}{{\mathbb W}}
 \nc{\bbX}{{\mathbb X}}
 \nc{\bbY}{{\mathbb Y}}
 \nc{\bbZ}{{\mathbb Z}}
 \nc{\bba}{{\mathbb a}}
 \nc{\bbb}{{\mathbb b}}
 \nc{\bbc}{{\mathbb c}}
 \nc{\bbd}{{\mathbb d}}
 \nc{\bbe}{{\mathbb e}}
 \nc{\bbf}{{\mathbb f}}
 \nc{\bbg}{{\mathbb g}}
 \nc{\bbh}{{\mathbb h}}
 \nc{\bbi}{{\mathbb i}}
 \nc{\bbk}{{\mathbb k}}
 \nc{\bbl}{{\mathbb l}}
 \nc{\bbm}{{\mathbb m}}
 \nc{\bbn}{{\mathbb n}}
 \nc{\bbo}{{\mathbb o}}
 \nc{\bbp}{{\mathbb p}}
 \nc{\bbq}{{\mathbb q}}
 \nc{\bbr}{{\mathbb r}}
 \nc{\bbs}{{\mathbb s}}
 \nc{\bbt}{{\mathbb t}}
 \nc{\bbu}{{\mathbb u}}
 \nc{\bbv}{{\mathbb v}}
 \nc{\bbw}{{\mathbb w}}
 \nc{\bbx}{{\mathbb x}}
 \nc{\bby}{{\mathbb y}}
 \nc{\bbz}{{\mathbb z}}

 \nc{\MZV}{{\mathcal{MZV}}}
 \nc{\calA}{{\mathcal A}}
 \nc{\calB}{{\mathcal B}}
 \nc{\calC}{{\mathcal C}}
 \nc{\calD}{{\mathcal D}}
 \nc{\calE}{{\mathcal E}}
 \nc{\calF}{{\mathcal F}}
 \nc{\calG}{{\mathcal G}}
 \nc{\calH}{{\mathcal H}}
 \nc{\calI}{{\mathcal I}}
 \nc{\calJ}{{\mathcal J}}
 \nc{\calK}{{\mathcal K}}
 \nc{\calL}{{\mathcal L}}
 \nc{\calM}{{\mathcal M}}
 \nc{\calN}{{\mathcal N}}
 \nc{\calO}{{\mathcal O}}
 \nc{\calP}{{\mathcal P}}
 \nc{\calQ}{{\mathcal Q}}
 \nc{\calR}{{\mathcal R}}
 \nc{\calS}{{\mathcal S}}
 \nc{\calT}{{\mathcal T}}
 \nc{\calU}{{\mathcal U}}
 \nc{\calV}{{\mathcal V}}
 \nc{\calW}{{\mathcal W}}
 \nc{\calX}{{\mathcal X}}
 \nc{\calY}{{\mathcal Y}}
 \nc{\calZ}{{\mathcal Z}}
  \nc{\cala}{{\mathcal a}}
 \nc{\calb}{{\mathcal b}}
 \nc{\calc}{{\mathcal c}}
 \nc{\cald}{{\mathcal d}}
 \nc{\cale}{{\mathcal e}}
 \nc{\calf}{{\mathcal f}}
 \nc{\calg}{{\mathcal g}}
 \nc{\calh}{{\mathcal h}}
 \nc{\cali}{{\mathcal i}}
 \nc{\calj}{{\mathcal j}}
 \nc{\calk}{{\mathcal k}}
 \nc{\call}{{\mathcal l}}
 \nc{\calm}{{\mathcal m}}
 \nc{\caln}{{\mathcal n}}
 \nc{\calo}{{\mathcal o}}
 \nc{\calp}{{\mathsf p}}
 \nc{\calq}{{\mathcal q}}
 \nc{\calr}{{\mathcal r}}
 \nc{\cals}{{\mathcal s}}
 \nc{\calt}{{\mathcal t}}
 \nc{\calu}{{\mathcal u}}
 \nc{\calv}{{\mathcal v}}
 \nc{\calw}{{\mathcal w}}
 \nc{\calx}{{\mathcal x}}
 \nc{\caly}{{\mathcal y}}
 \nc{\calz}{{\mathcal z}}

 \nc{\frakA}{{\mathfrak A}}
 \nc{\frakB}{{\mathfrak B}}
 \nc{\frakC}{{\mathfrak C}}
 \nc{\frakD}{{\mathfrak D}}
 \nc{\frakE}{{\mathfrak E}}
 \nc{\frakF}{{\mathfrak F}}
 \nc{\frakG}{{\mathfrak G}}
 \nc{\frakH}{{\mathfrak H}}
 \nc{\frakI}{{\mathfrak I}}
 \nc{\frakJ}{{\mathfrak J}}
 \nc{\frakK}{{\mathfrak K}}
 \nc{\frakL}{{\mathfrak L}}
 \nc{\frakM}{{\mathfrak M}}
 \nc{\frakN}{{\mathfrak N}}
 \nc{\frakO}{{\mathfrak O}}
 \nc{\frakP}{{\mathfrak P}}
 \nc{\frakQ}{{\mathfrak Q}}
 \nc{\frakR}{{\mathfrak R}}
 \nc{\frakS}{{\mathfrak S}}
 \nc{\frakT}{{\mathfrak T}}
 \nc{\frakU}{{\mathfrak U}}
 \nc{\frakV}{{\mathfrak V}}
 \nc{\frakW}{{\mathfrak W}}
 \nc{\frakX}{{\mathfrak X}}
 \nc{\frakY}{{\mathfrak Y}}
 \nc{\frakZ}{{\mathfrak Z}}
 \nc{\fraka}{{\mathfrak a}}
 \nc{\frakb}{{\mathfrak b}}
 \nc{\frakc}{{\mathfrak c}}
 \nc{\frakd}{{\mathfrak d}}
 \nc{\frake}{{\mathfrak e}}
 \nc{\frakf}{{\mathfrak f}}
 \nc{\frakg}{{\mathfrak g}}
 \nc{\frakh}{{\mathfrak h}}
 \nc{\fraki}{{\mathfrak i}}
 \nc{\frakj}{{\mathfrak j}}
 \nc{\frakk}{{\mathfrak k}}
 \nc{\frakl}{{\mathfrak l}}
 \nc{\frakm}{{\mathfrak m}}
 \nc{\frakn}{{\mathfrak n}}
 \nc{\frako}{{\mathfrak o}}
 \nc{\frakp}{{\mathfrak p}}
 \nc{\frakq}{{\mathfrak q}}
 \nc{\frakr}{{\mathfrak r}}
 \nc{\fraks}{{\mathfrak s}}
 \nc{\frakt}{{\mathfrak t}}
 \nc{\fraku}{{\mathfrak u}}
 \nc{\frakv}{{\mathfrak v}}
 \nc{\frakw}{{\mathfrak w}}
 \nc{\frakx}{{\mathfrak x}}
 \nc{\fraky}{{\mathfrak y}}
 \nc{\frakz}{{\mathfrak z}}
 \nc{\so}{{\mathfrak so}}
 \nc{\sa}{{\mbox{{\scriptsize \cyr x}}}}
 \nc{\slfour}{{\mathfrak sl}_4}
 \nc{\one}{{\bf 1}}
 \nc{\zero}{{\bf 0}}
 \nc{\Qxy}{\Q\langle x,y\rangle}

\begin{document}

\title[A Family of MHS and MZSV Identities]{A Family of Multiple Harmonic Sum and \\
Multiple Zeta Star Value Identities}

\author{Erin Linebarger}
\address{Department of Mathematics, Eckerd College, St. Petersburg, FL 33711}
\author{Jianqiang Zhao}
\address{Kavli Institute for Theoretical Physics China, Beijing, China,
Max Planck Institute for Mathematics, Vivatsgasse 7, 53111 Bonn, Germany, and
Department of Mathematics, Eckerd College, St. Petersburg, FL 33711}

\email{emlineba@eckerd.edu}

\email{zhaoj@eckerd.edu}


\date{}


\begin{abstract} In this paper we present a new family of identities for multiple harmonic sums
which generalize a recent result of  Hessami Pilehrood et al.\ \cite{HessamiPilehrood2Tauraso2012}.
We then apply it to obtain a family of identities relating multiple zeta star values to alternating
Euler sums. In such a typical identity the entries of the multiple zeta star values
consist of blocks of arbitrarily long 2-strings separated by positive integers greater than two
while the largest depth of the alternating Euler sums depends only on the number of 2-string blocks
but not on their lengths.
\end{abstract}

\maketitle

\section{Introduction}\label{intro}
In a recent paper \cite{HessamiPilehrood2Tauraso2012} Hessami Pilehrood et al.\ proved a few families of
identities involving (alternating) multiple harmonic sums and binomial coefficients and, as applications,
discovered some new congruences for multiple harmonic sums. In particular, they are able to confirm several
conjectures contained in \cite{Zhao2008a} posed by the second author of this paper.
They can also provide some new families of identities of MZSV
and a new proof of the identity of Zagier \cite{Zagier2012}.

To state one of their results we recall that the (alternating) multiple harmonic sums are multiple variable
generalizations of harmonic sums defined by the following:
for any $r\in \N$ and $\bfs=(s_1, s_2, \ldots, s_r)\in  (\Z^*)^r$
\begin{equation*}
H_n(s_1, s_2, \ldots, s_r)=\sum_{n\ge k_1>k_2>\ldots>k_r\ge 1}\prod_{i=1}^r
\frac{\sgn (s_i)^{k_i}}{k_i^{|s_i|}}.
\end{equation*}
Throughout the paper we will use $\bar{n}$ to denote a negative entry $s_j=-n$. For example, $H(\bar 2,1)=H(-2,1).$
Another kind of sums which is closely
related to the above is the following star version (also denoted by S in the literature)
\begin{equation*}
H^\star_n(s_1, s_2, \ldots, s_r)=\sum_{n\ge k_1\ge k_2 \ge \ldots \ge k_r\ge 1}\prod_{i=1}^r
\frac{\sgn (s_i)^{k_i}}{k_i^{|s_i|}}.
\end{equation*}
We call $l(\bfs):=r$ the depth of this MHS and $|\bfs|:=\sum_{i=1}^r|s_i|$ the weight.
For convenience we set  $H_n(\bfs)=0$ if $n<l(\bfs)$,
$H_n(\emptyset)=H^\star_n(\emptyset)=1$, and $\{s_1, s_2, \ldots, s_j\}^d$
the set formed by repeating the composition $(s_1, s_2, \ldots, s_j)$ $d$ times. When taking
the limit $n\to\infty$ we get the so-called the \emph{(alternating) Euler sum} and
the \emph{(alternating) star Euler sum}, respectively:
\begin{equation}\label{equ:EulerSumDefn}
\zeta(\bfs)=\lim_{n\to\infty} H_n(\bfs),\qquad
\zeta^\star(\bfs)=\lim_{n\to\infty} H^\star_n(\bfs).
\end{equation}
When $\bfs\in\N^\ell$ they are called the \emph{multiple zeta value} (MZV)
and the \emph{multiple zeta star value} (MZSV), respectively.

The following is one of the main results of \cite{HessamiPilehrood2Tauraso2012}.
\begin{thm} \label{thm:PTThm2.1} {\rm (\cite[Thm.~2.1]{HessamiPilehrood2Tauraso2012})}
Let $a$ and $b$ be two non-negative integers. Set
$A_{n,k}=(-1)^{k-1}\binom{n}{k}/\binom{n+k}{k}$ for any positive integers $n$ and $k$.
Then for any integer $c\ge 2$,
\begin{equation*}
H^\star_n(\{2\}^a,c,\{2\}^b)=
2\sum_{k=1}^n \frac{A_{n,k}}{k^{2a+2b+c}}
+4 \sum_{\substack{i+j+|\bfx|=c\\ i\ge 1,j\ge 2,\bfx\in\N^r,r\ge0}}
2^{l(\bfx)}
\sum_{k=1}^n\frac{H_{k-1}(\bfx,i+2b)A_{n,k}}{k^{2a+j}},
\end{equation*}
where in the sum $i\ge 1,j\ge 2$ and $\bfx$ runs through all possible compositions of positive
integers (or the empty set when $r = 0$) satisfying the restriction $i + j + |\bfx| = c$.
\end{thm}

In sections \ref{sec:MHS} and \ref{sec:proofMHS} we generalize Theorem~\ref{thm:PTThm2.1} to express
$H^\star_n(\bfs)$ using binomial coefficients where
$\bfs=(\{2\}^{a_1},c_1,\dots,\{2\}^{a_r},c_r,\{2\}^{a_{r+1}})$
with positive integers $c_1,\dots,c_r\ge 2$. Taking limit and
using a key lemma we shall obtain a new family of identities of MZSV
involving alternating Euler sums in the following Theorem (an equivalent
form of Theorem~\ref{thm:GenThm2.1MZSV}).
\begin{thm}\label{thm:IntroGenThm2.1MZSV}
Let $n$ be a positive integer and
$\bfs=(\{2\}^{a_1},c_1,\dots,\{2\}^{a_r},c_r,\{2\}^{a_{r+1}})$
where $a_j, c_j\in\N_0$ and $c_j\ge 3$.
Then  we have
\begin{equation*}
 \zeta^\star(\bfs)= -\sum_{\bfp=(\overline{2a_1+2})\circ 1^{\circ(c_1-3)} \circ (2 a_2+3) \circ 1^{\circ(c_2-3)}
 \circ \dots\circ  (2 a_r+3) \circ 1^{\circ(c_r-3)} \circ (2 a_{r+1}+1)  }2^{\ell(\bfp)}\zeta(\bfp)
\end{equation*}
where $\circ$ is either the comma ``,'' or the O-plus ``$\oplus$'' defined by
$\alpha \oplus \beta=\sgn(\alpha )\sgn(\beta)(|\alpha |+|\beta|)$ for all $\alpha,\beta\in\Z^*$, and $1^{\circ c}=\underbrace{1 \circ \dots\circ1}_{c \ {\rm times}}$.
\end{thm}
For example,
\begin{align*}
\zeta^\star(\{2\}^7,3,\{2\}^2,3,\{2\}^3)=&-8\zeta(\ol{16},7,7)
    -4\zeta(\ol{23},7)-4\zeta(\ol{16},14) -2\zeta(\ol{30}),\\
\zeta^\star(\{2\}^3,3,\{2\}^2,4,\{2\}^5)=&-16\zeta(\ol{8},7,1,11)-8\zeta(\ol{8},7,12)
    -8\zeta(\ol{8},8,11)\\
-8\zeta(\ol{15},1,11)&-4\zeta(\ol{15},12)-4\zeta(\ol{16},11)-4\zeta(\ol{8},19)-2\zeta(\ol{27}).
\end{align*}
Broadhurst considered alternating Euler sums as ``honorary''
MZVs and relates them to the study of knots and Feynman diagrams
(see \cite{Broadhurst1996}. From the point of view of algebraic geometry,
Goncharov and Manin \cite{GoncharovMa2004} showed that the MZVs and MZSVs are periods of mixed Tate motives over $\Z$
(see also \cite{Brown2012}) while Deligne and Goncharov \cite{DeligneGo2005} proved that
the alternating Euler sums are periods of mixed Tate motives over $\Z$ ramified at 2
(see also \cite{Deligne2010}). This provides a kind of inverse to ``Galois descent' since the linear
combinations of the Euler sums in the above examples are invariant under a certain motivic Galois group.
In general it is not easy to show this unramifiedness property directly. Another important aspect of
Theorem~\ref{thm:IntroGenThm2.1MZSV} is that the depth of the MZSV
can be very large when the 2-strings are very long, but
the largest depth of the Euler sums on the right hand side depends only on the number of
such strings but not on their lengths.

The ideas used in this paper have been applied to obtain similar results for other
types of strings in \cite{Zhao2013}. In particular, the conjectural Two-one formula
of Ohno and Zudlin is proved there. Its $q$-analog has been obtained by Hessami Pilehroods recently in \cite{HessamiPilehrood22013}.
In \cite{HessamiPilehrood2Zhao2013} Hessami Pilehroods and the second author further prove a $q$-analog of the main result of this paper (Theorem~\ref{thm:GenThm2.1}) and some similar results for other types of strings contained in \cite{Zhao2013}.

\section{Identities for multiple harmonic sums}\label{sec:MHS}

The following lemma is the special case of \cite[Lemma 2.2]{HessamiPilehrood2Tauraso2012} when $m=2$ and $c_n^{(2)}=n!^2/(2n)!$.

\begin{lem} \label{lem:PTlemma2.1}
For any positive integers $n$ and $k$ define $A_{n,k}=(-1)^{k-1}\binom{n}{k}/\binom{n+k}{k}$.
Then for every $c\in \N$ and
$\bfv\in \N^t$ ($t\ge 0$) we have
\begin{equation*}
\frac{1}{n^c}\sum_{k=1}^n \frac{H_{k-1}(\bfv)A_{n,k}}{k^a}
=\sum_{k=1}^n \frac{H_{k-1}(\bfv)A_{n,k}}{k^{a+c}}
+\sum_{\substack{j+|\bfx|=a+c\\ j\ge 0, x_r>a}} 2^{l(\bfx)}
\sum_{k=1}^n\frac{H_{k-1}(\bfx,\bfv)A_{n,k}}{k^{j}}
\end{equation*}
where $\bfx=(x_1,x_2,\dots,x_r)\in \mathbb{N}^r$ for $r\ge 0$, and $|\bfx|=\sum_{i=1}^r x_i$, $l(\bfx)=r$.
Here $|\bfx|=0$ if $r=0$ which implies that $\bfx=\emptyset$.
\end{lem}

\begin{defn}\label{defn:notation}
Suppose $\bfs=(\{2\}^{a_1},c_1,\dots,\{2\}^{a_r},c_r,\{2\}^{a_{r+1}})$.
We define two kinds of string operations on the \emph{condensation}
$\tbfs:=(2 a_1,c_1,\dots,2 a_r,c_r,2 a_{r+1})$ of $\bfs$ as follows.
\begin{itemize}
  \item A \emph{merge} $\mu_t$ for some $1\le t\le r$ changes the symbols
``$,c_t,$'' to ``$+c_t+$'',
  \item A \emph{substitution} $\gs_t$ changes the symbols
``$,c_t,$'' to ``$+j_t, \bfx_t, i_t+$''.
\end{itemize}
Let $I=\{i_1,\dots,i_t\}$ be a subset of $[r]:=\{1,\dots,r\}$. Define
\begin{itemize}
  \item The merge operation $\mu_I=\mu_{i_1}\circ\cdots \circ\mu_{i_t}$ and similarly for substitutions $\gs_I$,
  \item The complement $\oI=[r]\setminus I$,
  \item The combined operation $\gk_I=\gs_\oI\circ\mu_I$.
\end{itemize}
For any composition $(e_1,\dots,e_r)$ of natural numbers we denote
\begin{itemize}
  \item The first component $e_1=\gf(\bfe)$
  \item The truncation operator $\Trunc(e_1,\dots,e_r)=(e_2,\dots,e_r)$.
  \item The negation operator $\nu(e_1,\dots,e_r)=(\ol{e_1},e_2,\dots,e_r)$.
\end{itemize}
Notice that if $r=1$ then $\Trunc(e_1)=\emptyset$.
\end{defn}

\begin{thm} \label{thm:GenThm2.1}
Let $n$ be a positive integer and
$\bfs=(\{2\}^{a_1},c_1,\dots,\{2\}^{a_r},c_r,\{2\}^{a_{r+1}})$
where $a_j, c_j\in\N_0$ and $c_j\ge 2$. Set $A_{n,k}=(-1)^{k-1}\binom{n}{k}/\binom{n+k}{k}$.
Then
\begin{equation}\label{equ:GenMain}
 H^\star_n(\bfs)=  2\sum_{I\subseteq [r]}\
\sum_{\substack{\ i_t+j_t+|\bfx_t|=c_t, \\ i_t\ge 1,j_t\ge 2 \ \forall t\not\in I }}
 2^{|\oI|+\underset{t\not\in I}{\sum} l(\bfx_t)} \sum_{k=1}^n\frac{H_{k-1}\Big(\Trunc\circ\gk_I(\tbfs) \Big)A_{n,k}}
 {k^{\gf\circ\gk_I(\tbfs)} },
\end{equation}
where $\tbfs=(2 a_1,c_1,\dots,2 a_r,c_r,2 a_{r+1})$ is the condensation of $\bfs$.
\end{thm}

\begin{eg}
When $r=1$ we have two possible subsets of $[1]$: $I=\emptyset$ and $I=[1]$.
our Theorem~\ref{thm:GenThm2.1} becomes Theorem~\ref{thm:PTThm2.1}. When $r=2$ we get the following:
For a positive integer $n$ and non-negative integers $a_1, a_2, a_3$ and
and positive integers $c_1, c_2\ge 2$ we set $\bfs=(\{2\}^{a_1},c_1,(\{2\}^{a_2}, c_2,(\{2\}^{a_3})$. Then
\begin{align*}
\gk_{[2]} =   (2a_1 + c_1 +2a_2 + c_2 +2a_3),\quad
\gk_\emptyset =   (2a_1 + j_1, \bfx_1, i_1 +2a_2 + j_2, \bfx_2, i_2 +2a_3), \\
\gk_{\{1\}}=(2a_1 + c_1 +2a_2 + j_2, \bfx_2, i_2 +2a_3), \quad
\gk_{\{2\}}=(2a_1 + j_1, \bfx_1, i_1,2a_2 + c_2 +2a_3).
\end{align*}
Hence by Theorem~\ref{thm:GenThm2.1}
\begin{align*}
H^\star_n(\bfs) &  = 2\sum_{k=1}^n \frac{A_{n,k}}{k^{2a_1+c_1+2a_2+c_2+2a_3}}
+4\sum_{\substack{i+j+|\bfx|=c_2\\ i\ge 1,j\ge 2 }} 2^{l(\bfx)}
\sum_{k=1}^n\frac{H_{k-1}(\bfx,i+2a_3)A_{n,k}}{k^{2a_1+c_1+2a_2+j}}     \\
&+ 8\sum_{\substack{ i_1+j_1+|\bfx_1|=c_1,i_2+j_2+|\bfx_2|=c_2 \\ i_1,i_2\ge 1,j_1,j_2\ge 2}} 2^{l(\bfx_1)+l(\bfx_2)}
\sum_{k=1}^n\frac{H_{k-1}(\bfx_1,i_1+2a_2+j_2,\bfx_2,i_2+2a_3)A_{n,k}}{k^{2a_1+j_1}}  \\
&  +4\sum_{\substack{i+j+|\bfx|=c_1\\ i\ge 1,j\ge 2}}2^{l(\bfx)}
\sum_{k=1}^n\frac{H_{k-1}(\bfx,i+2a_2+c_2+2a_3)A_{n,k}}{k^{2a_1+j}} .
\end{align*}
\end{eg}

\section{Proof of Theorem \ref{thm:GenThm2.1}}\label{sec:proofMHS}
We proceed by induction on the sum $r + n$. Since $r \ge 0, n \ge 1$, our base case is $r + n = 1$
which implies $H^\star_1(\{2\}^a) = 1$, and therefore the formula is true.
Suppose the statement is true when $r + n =k-1$, and let $r + n = k$. Then by definition
\begin{equation*}
H^\star_n(\bfs) = \sum_{l=0}^{a_1} \frac {1}{n^{2(a_1-l)}} H^\star_{n-1}(\bft_l)
+ \frac {1}{n^{2a_1+ c_1}} H^\star_n(\bfu),
\end{equation*}
where $\bft_l=(\{2\}^l, c_1, \{2\}^{a_2},c_2, \dots,c_r, \{2\}^{a_{r+1}})$ and
$\bfu=(\{2\}^{a_2}, c_2, \dots, c_{r}, \{2\}^{a_{r+1}})$. We now set their condensations
as follows:
\begin{equation*}
\tbft_l=(2l, c_1, 2a_2,c_2, \dots,c_r,2a_{r+1}),
\quad \tbfu=(2a_2, c_2, \dots, c_{r}, 2a_{r+1}).
\end{equation*}
By inductive hypothesis,
\begin{align*}
H^\star_n(\textbf{s})=&
 \sum_{l=0}^{a_1} \frac {2}{n^{2(a_1-l)}} \sum_{I\subseteq [r]}\
 \sum_{\substack{\ i_t+j_t+|\bfx_t|=c_t, \\ i_t \ge 1, j_t\ge 2 \ \forall t\not\in I }} 2^{|\oI|+
 \underset{t\not\in I}{\sum} l(\bfx_t)}
 \sum_{k=1}^{n-1}\frac{H_{k-1}\Big(\Trunc\circ\gk_I(\tbft_l)\Big)A_{n-1,k}}
 {k^{\gf\circ\gk_I(\tbft_l)}  }\\
+&  \frac {2}{n^{2a_1+ c_1}} \sum_{I\subseteq [r]\setminus\{1\}}\
 \sum_{\substack{\ i_t+j_t+|\bfx_t|=c_t, \\ i_t \ge 1 ,j_t\ge 2 \ \forall t\not\in I }}
 2^{|\oI|+\underset{t\not\in I}{\sum} l(\bfx_t)}
 \sum_{k=1}^n\frac{H_{k-1}\Big(\Trunc\circ\gk_I(\tbfu)\Big)A_{n,k}}
 {k^{\gf\circ\gk_I(\tbfu)} }.
\end{align*}
Changing the order of summation and summing the inner sum we get
\begin{equation*}
A_{n-1,k} \sum_{l=0}^{a_1} \frac{n^{2l}}{k^{2l}} = (-1)^{k-1} \frac{n^{2a_1+ 2} - k^{2a_1+ 2}}{(n^2 - k^2)k^{2a_1}} \frac {\binom{n-1}{k}}{\binom{n-1+k}{k}} = \left (\frac{n^{2a_1}}{k^{2a_1}} - \frac{k^2}{n^2} \right )A_{n,k} .
\end{equation*}
Thus
{\allowdisplaybreaks
\begin{align*}
 H^\star_n(\bfs) =&2 \sum_{I\subseteq [r]}\
 \sum_{\substack{\ i_t+j_t+|\bfx_t|=c_t, \\ i_t\ge 1 ,j_t\ge 2 \ \forall t\not\in I }}
 2^{|\oI|+\underset{t\not\in I}{\sum} l(\bfx_t)}
 \sum_{k=1}^{n}\frac{H_{k-1}\Big(\Trunc\circ\gk_I(\tbft_l)\Big)A_{n,k}}
 {k^{\gf\circ\gk_I(\tbft_l) - 2l+ 2a_1} }\\
-&\frac{2}{n^{2a_1+2}} \sum_{I\subseteq [r]}\
 \sum_{\substack{\ i_t+j_t+|\bfx_t|=c_t, \\ i_t \ge 1,j_t\ge 2 \ \forall t\not\in I }}
 2^{|\oI|+\underset{t\not\in I}{\sum} l(\bfx_t)}
 \sum_{k=1}^{n}\frac{H_{k-1}\Big(\Trunc\circ\gk_I(\tbft_l)\Big)A_{n,k}}
 {k^{\gf\circ\gk_I(\tbft_l) - 2l - 2} }\\
+& \frac {2}{n^{2a_1+c_1}} \sum_{I\subseteq [r]\setminus\{1\}}\
 \sum_{\substack{\ i_t+j_t+|\bfx_t|=c_t, \\ i_t\ge 1,j_t\ge 2 \ \forall t\not\in I }}
 2^{|\oI|+\underset{t\not\in I}{\sum} l(\bfx_t)}
 \sum_{k=1}^{n}\frac{H_{k-1}\Big(\Trunc\circ\gk_I(\tbfu)\Big)A_{n,k}}
 {k^{\gf\circ\gk_I(\tbfu)} }.
\end{align*}}
Since $\tbft_l$ coincides with $\tbfs$ everywhere except for the first component, $\Trunc\circ\gk_I(\tbft_l) =
\Trunc\circ\gk_I(\tbfu)$. Similarly, $\gf\circ\gk_I(\tbft_l) - 2l+ 2a_1= \gf\circ\gk_I(\tbfu)$.
Hence, by induction we only need to show the following quantity vanishes:
{\allowdisplaybreaks
\begin{align}
 &\frac {2}{n^{2a_1+c_1}} \sum_{I\subseteq [r]\setminus\{1\}}\
 \sum_{\substack{\ i_t+j_t+|\bfx_t|=c_t, \\ i_t\ge 1,j_t\ge 2 \ \forall t\not\in I }}
 2^{|\oI|+\underset{t\not\in I}{\sum} l(\bfx_t)}
 \sum_{k=1}^{n}\frac{H_{k-1}\Big(\Trunc\circ\gk_I(\tbfu)\Big)A_{n,k}}
 {k^{\gf\circ\gk_I(\tbfu)} }\notag\\
-& \frac{2}{n^{2a_1+2}} \sum_{I\subseteq [r]}\
 \sum_{\substack{\ i_t+j_t+|\bfx_t|=c_t, \\ i_t \ge 1,j_t\ge 2 \ \forall t\not\in I }}
 2^{|\oI|+\underset{t\not\in I}{\sum} l(\bfx_t)}
 \sum_{k=1}^{n}\frac{H_{k-1}\Big(\Trunc\circ\gk_I(\tbft_l)\Big)A_{n,k}}
 {k^{\gf\circ\gk_I(\tbft_l) - 2l - 2} }\notag\\
=& \frac {2}{n^{2a_1+c_1}} \sum_{I\subseteq [r]\setminus\{1\}}\
 \sum_{\substack{\ i_t+j_t+|\bfx_t|=c_t, \\ i_t\ge 1,j_t\ge 2 \ \forall t\not\in I }}
 2^{|\oI|+\underset{t\not\in I}{\sum} l(\bfx_t)}
 \sum_{k=1}^{n}\frac{H_{k-1}\Big(\Trunc\circ\gk_I(\tbfu)\Big)A_{n,k}}
 {k^{\gf\circ\gk_I(\tbfu)} }\notag\\
-& \frac{2}{n^{2a_1+2}} \sum_{I\subseteq [r]\setminus\{1\}}\
 \sum_{\substack{\ i_t+j_t+|\bfx_t|=c_t, \\ i_t \ge 1,j_t\ge 2 \ \forall t\not\in I }}
 2^{|\oI|+\underset{t\not\in I}{\sum} l(\bfx_t)}
 \sum_{k=1}^{n}\frac{H_{k-1}\Big(\Trunc\circ\gk_I(\tbft_l)\Big)A_{n,k}}
 {k^{\gf\circ\gk_I(\tbft_l) - 2l - 2} }\notag\\
-& \frac{4}{n^{2a_1+2}} \sum_{I\subseteq [r]\setminus\{1\}}\
 \sum_{\substack{\ i_t+j_t+|\bfx_t|=c_t, \\ i_t\ge 1,j_t\ge 2 \ \forall t\not\in I }}
 2^{|\oI|+\underset{t\not\in I}{\sum} l(\bfx_t)}
 \sum_{k=1}^{n}\frac{H_{k-1}\Big(\Trunc\circ\gk_{I\cup\{1\}}(\tbft_l)\Big)A_{n,k}}
 {k^{\gf\circ\mu_{\bI}\circ\gs_{I \cup \{1\}}(\tbft_l) - 2l - 2} } \notag\\
=&\frac{2}{n^{2a_1+2}} \sum_{I\subseteq [r]\setminus\{1\}}\
 \sum_{\substack{\ i_t+j_t+|\bfx_t|=c_t, \\ i_t\ge 1,j_t\ge 2 \ \forall t\not\in I }}
 2^{|\oI|+\underset{t\not\in I}{\sum} l(\bfx_t)}
 \Bigg( \frac{1}{n^{c_1-2}} \sum_{k=1}^{n}\frac{H_{k-1}\Big(\Trunc\circ\gk_I(\tbfu)\Big)A_{n,k}} {k^{\gf\circ\gk_I(\tbfu)} }\notag\\
-&  2\sum_{\substack{\ i_1+j_1+|\bfx_1|=c_1, \\ i_1 \ge 1,j_1\ge 2 }} 2^{l(\bfx_1)}
\sum_{k=1}^{n}\frac{H_{k-1}\Big(\Trunc\circ\gk_I(\tbft_l)\Big)A_{n,k}}
 {k^{\gf\circ\gk_I(\tbft_l) - 2l - 2} }
 - \sum_{k=1}^{n}\frac{H_{k-1}\Big(\Trunc\circ\gk_{I\cup\{1\}}(\tbft_l)\Big)A_{n,k}}
 {k^{\gf\circ\gk_{I \cup \{1\}}(\tbft_l) - 2l - 2} } \Bigg) . \label{equ:last3terms}
\end{align}}
Note that if $1\not\in I$ then $\Trunc\circ\gk_I(\tbfu) = \Trunc\circ\gk_{I \cup \{1\}}(\tbft_l)$, which we will now denote by $\bfv$. Furthermore, $\gf\circ\gk_{I \cup \{1\}}(\tbft_l)=\gf\circ\gk_I(\tbfu)+c_1+2l$. Let $a = \gf\circ\gk_I(\tbfu)$. Then
\begin{equation*}
\gk_I(\tbft_l) =\gs_{\oI}\circ\mu_I(\tbft_l)=\gs_1(2l,c_1,a,\dots)=(j_1+2l,\bfx_1, i_1+a,\dots).
\end{equation*}
Set $\bfw=(\bfx_1, i_1+a)$. Then we see that the three sums inside the parenthesis \eqref{equ:last3terms}
can be rewritten as
\begin{multline*}
 \frac{1}{n^{c_1-2}} \sum_{k=1}^{n}\frac{H_{k-1}(\bfv)A_{n,k}} {k^a }
- \sum_{\substack{\ j_1+|\bfw|=c_1+a_1, \\ j_1\ge 2, w_p > a_1 }} 2^{l(\bfw)}
 \sum_{k=1}^{n}\frac{H_{k-1} (\bfw, \bfv)A_{n,k}}
 {k^{j_1-2} }
 - \sum_{k=1}^{n}\frac{H_{k-1}(\bfv)A_{n,k}}
 {k^{a + c_1 -2} },
\end{multline*}
where $w_p=i+a_1$ ($i\ge 1$) is the last component of $\bfw$. This last expression vanishes by taking $c=c_1-2$
and $j=j_1-2$ in Lemma \ref{lem:PTlemma2.1}.
We have now completed the proof of Theorem~\ref{thm:GenThm2.1}.

\section{The key lemma and an identity family of MZSV}\label{sec:MZSV}
In \cite{HessamiPilehrood2Tauraso2012}, using the corresponding identities of MHS
Hessami Pilehrood et al.\ find some new families of identities of MZSV
and a new proof of the identity of Zagier \cite{Zagier2012}.
We can similarly use Theorem~\ref{thm:GenThm2.1} to obtain a new family of MZSV as follows.
\begin{thm}\label{thm:GenThm2.1MZSV}
Let $n$ be a positive integer and
$\bfs=(\{2\}^{a_1},c_1,\dots,\{2\}^{a_r},c_r,\{2\}^{a_{r+1}})$
where $a_j, c_j\in\N_0$ and $c_j\ge 2$. Set
$\tbfs=(2 a_1,c_1,\dots,2 a_r,c_r,2 a_{r+1})$.
Then with notation given by Definition~\ref{defn:notation} we have
\begin{equation}\label{equ:GenMainMZSV}
 \zeta^\star(\bfs)=   2\sum_{I\subseteq [r]}\
\sum_{\substack{\ i_t+j_t+|\bfx_t|=c_t, \\ i_t\ge 1,j_t\ge 2 \ \forall t\not\in I }}
 2^{|\oI|+\underset{t\not\in I}{\sum} l(\bfx_t)} \zeta\Big(\nu\circ\gk_I(\tbfs) \Big).
\end{equation}
\end{thm}
\begin{proof}
By taking limit $n\to \infty$ in Theorem~\ref{thm:GenThm2.1} we see that \eqref{equ:GenMainMZSV}
follows immediately from the following lemma.
\end{proof}

\begin{lem} \label{lem:limitBound}
Let $d\in\N_0$ and let $e$ be a real number with $e>1$. Then for all $\bfs\in(\Z^*)^d$ ($\bfs=\emptyset$ if $d=0$) we have
 \begin{equation}\label{equ:limit0}
\lim_{n\to\infty} \sum_{k=1}^n\frac{|H_{k-1} (\bfs)|}
 {k^e} \left(1-\frac{\binom{n}{k}}{\binom{n+k}{k}}\right)=0.
\end{equation}
\end{lem}

\begin{rem}
This is the key step which enables us to go from MHS identities to MZSV identities in this paper. Apparently the authors of \cite{HessamiPilehrood2Tauraso2012} have already used this result in their paper although no proof is given there. Because of its importance we provide the following detailed analysis. Notice that for any fixed $k$ we have $\lim_{n\to\infty} \binom{n}{k}/\binom{n+k}{k} =1$ but for $k$ close to $n$, say $k=n$, we have
$\lim_{n\to\infty}  \binom{n}{k}/\binom{n+k}{k} =0$ by Stirling's formula. So \eqref{equ:limit0} is not obvious to us.
\end{rem}

\begin{proof}
First we have
\begin{align}
0< &1-\frac{\binom{n}{k}}{\binom{n+k}{k}}
=1-\frac{n(n-1)(n-2)\cdots (n-k+1)}{(n+1)(n+2)\cdots(n+k)} \notag \\
=&1-\frac{n(n-1)(\frac{n}2-1)\cdots (\frac{n}{k-1}-1)}{k(n+1)(\frac{n}2+1)\cdots (\frac{n}k+1)} \notag \\
=& 1-\frac{n^kH_{k-1}(\{1\}^{k-1})-n^{k-1}H_{k-1}(\{1\}^{k-2})+\cdots+(-1)^{k} n^2H_{k-1}(1)-n(-1)^{k}} {k\big(n^kH_k(\{1\}^k)+n^{k-1}H_k(\{1\}^{k-1})+\cdots+nH_k(1)+1\big)} \label{equ:cancelDeno}\\
\le &  \frac{n^{k-1}H_k(\{1\}^{k-2})+n^{k-2}H_k(\{1\}^{k-3})+\cdots+ n^2 H_k(1)+n} {k(n^kH_k(\{1\}^k)+n^{k-1}H_k(\{1\}^{k-1})+\cdots+nH_k(1)+1)}  \label{equ:cancelDeno2}\\
+&\frac{n^{k-1}H_k(\{1\}^{k-1})+\cdots+nH_k(1)+1} {n^kH_k(\{1\}^k)+n^{k-1}H_k(\{1\}^{k-1})+\cdots+nH_k(1)+1}. \notag
\end{align}
Here from \eqref{equ:cancelDeno} to \eqref{equ:cancelDeno2} we canceled the leading term of the numerator in \eqref{equ:cancelDeno} because of $H_k(\{1\}^k)=1/k!$, took
absolute values of each term on the right and then changed all $H_{k-1}$ to $H_k$.
Observe that for any $1\le j\le k$ we have
\begin{align*}
   H_k(\{1\}^j)=&\sum_{k\ge n_1>\dots>n_j>0} \frac{1}{n_1 \cdots n_j} \\
   \le &\sum_{k\ge n_1>\dots>n_j>n_{j+1}>0} \frac{k}{n_1  \cdots n_j\cdot n_{j+1}}
   =kH_k(\{1\}^{j+1})\le\cdots\le k^{k-j}H_k(\{1\}^k),
\end{align*}
which is even true for $j=0$. Thus we get
\begin{align*}
    1-\frac{\binom{n}{k}}{\binom{n+k}{k}}\le
   \frac{2(n^{k-1}k+n^{k-2}k^2+\cdots+nk^{k-1}+k^k)}{n^k}
    \le \frac{2k^2}n.
\end{align*}
Now for any $\bfs\in(\Z^*)^d$ we clearly have
\begin{equation*}
|H_{k-1}(\bfs)| \le H_{k-1}(\{1\}^d)\le C\log^d (k)
\end{equation*}
for some positive constant $C$. Therefore by setting $\delta=\min\{(e-1)/2,1/2\}>0$ we have
\begin{equation}\label{equ:inequ}
\sum_{k=1}^n\frac{|H_{k-1}(\bfs)|}{k^e} \left(1-\frac{\binom{n}{k}}{\binom{n+k}{k}}\right)
 \le \frac{2C}{n} \sum_{k=1}^{\lfloor n^\delta\rfloor} \frac{\log^d (k)}{k^{e-2}} +  C\sum_{k=\lfloor n^\delta\rfloor+1}^{\infty} \frac{\log^d (k)}{k^e}.
\end{equation}
If $e\ge 2$ then
\begin{equation*}
 \frac{2C}{n} \sum_{k=1}^{\lfloor n^\delta\rfloor} \frac{\log^d (k)}{k^{e-2}}\le
 \frac{2C}{n} \sum_{k=1}^{\sqrt{n}} \log^d (k)  \le
\frac{2C\log^{d}(\sqrt{n})}{\sqrt{n}} \to 0  \text{ as $n\to \infty$.}
\end{equation*}
If $1<e< 2$ then
\begin{equation*}
 \frac{2C}{n} \sum_{k=1}^{\lfloor n^\delta\rfloor} \frac{\log^d (k)}{k^{e-2}}\le
 \frac{2C n^{\delta(3-e)}\log^{d}(n^\delta)}{n}\le  \frac{2C \delta^d \log^{d}(n)}{\sqrt{n}} \to 0 \text{ as $n\to \infty$},
\end{equation*}
since $1+\delta(e-3)=1+(e-1)(e-3)/2=(e-2)^2/2+1/2\ge 1/2$.
Now the last sum on the right of \eqref{equ:inequ} is the tail of a convergent series
so \eqref{equ:limit0} follows immediately and therefore the lemma is proved.
\end{proof}

\medskip
\noindent
\textbf{Acknowledgement}. We would like to thank Roberto Tauraso for sending us their
preprint \cite{HessamiPilehrood2Tauraso2012} upon which we can build our current work,
Wadim Zudilin for kindly showing us an improvement on Lemma~\ref{lem:limitBound}, and
K.\ Hessami Pilehrood, T.\ Hessami Pilehrood and Y.\ Ohno for pointing out some misprints in the
first version of this paper. We thank the anonymous referees who provided a few very insightful remarks
which are incorporated in the paper. JZ is supported by an NSF grant DMS1162116.

\end{document}